\documentclass[10pt,a4paper]{amsart}

\usepackage{setspace}
\usepackage[utf8x]{inputenc}
\usepackage[T1]{fontenc}

\usepackage{helvet}

\usepackage[geometry]{ifsym}
\usepackage{amsthm}
\usepackage{amsmath}
\usepackage{verbatim}
\usepackage{amssymb} 
\usepackage{resizegather}
\usepackage{tikz}
\usepackage{tikz-cd}
\usepackage{babel} 
\usepackage{mathrsfs}
\usepackage{graphicx}
\usepackage{epstopdf}
\usepackage{epigraph}

\usepackage{pgfplots}
\usepgfplotslibrary{fillbetween}
\pgfplotsset{compat=1.13}

\usepackage[bottom]{footmisc}

\usepackage{hyperref}

\newtheorem{theorem}{Theorem}

\newtheorem{proposition}{Proposition}[theorem]

\newtheorem{corollary}{Corollary}[theorem]

\newtheorem{remark}{Remark}

\newtheorem{definition}{Definition}

\title{On Hilbert's 16th Problem}
\author{Lars Andersen}
\begin{document}

\maketitle
\begin{abstract}
    We prove that to each real singularity $f: (\mathbb{R}^{n}, 0) \to (\mathbb{R}^k, 0)$ with $k\geq 2$ one can associate systems of differential equations $\mathfrak{g}^{k}_f$ which are pushforwards in the category of $\mathcal{D}$-modules over $\mathbb{R}^{k}$ of the sheaf of real analytic functions on the total space of the Milnor fibration. We then use this to study Hilbert's 16th problem on polynomial dynamical systems in the plane.
\end{abstract}
\begin{center}
    \text{Classification: }\textbf{14-XX, 94-XX}
\end{center}

\section{Introduction}

Since Isaac Newton's \cite{newton} unification of analysis, geometry and physics these subjects have tangented yet never united. A brief glimpse of their interplay resurfaced in Henri Poincarés invention of bifurcation theory where one studies the topological properties of the solutions of differential equations without having to find exact formulas for the solutions. This was further developed by the french and russian schools of geometers, in particular by Réné Thom and Vladimir Arnold and their students. On the geometric side was simultaneously Brieskorn and Milnor's work and in particular Milnor's study of the local topology of singularities. On the analytic side was, precedingly, Y. Manins work on the Gauss-Manin system of differential equations.\\

In more detail. Milnor discovered in \cite{Milnor} smooth fibrations associated with isolated singularities $(\mathbb{C}^{n+1}, 0)\to (\mathbb{C}, 0)$.  The homology groups of the fibres of these Milnor fibrations can be computed as Manin proved by solving the Gauss-Manin system. This was proven in 1958 (\cite{Man58}).\\ 
In a previous article the author defined similar systems of differential equations for real analytic singularities $f: (\mathbb{R}^{n+1}, 0)\to (\mathbb{R}, 0)$. In this article we will consider singularities $(\mathbb{R}^{n}, 0)\to (\mathbb{R}^k, 0)$ and apply these results to Hilbert's 16th problem.

\section*{Acknowledments}
    The author wishes to thank his late mother for her everlasting love and light and for telling me that everything is going to be alright. This is dedicated to her memory, and be it true or false it might raise questions and perhaps inspire further research on the local nature of analysis and geometry and in particular between singularities of varieties and the topology of dynamical systems. He also likes to thank his sister and father for their support throughout.

\subsection{A Real Milnor Fibration}
Suppose given a real analytic map germ $f: (\mathbb{R}^{n}, 0)\to (\mathbb{R}^k, 0)$ such that any representative of $(V(f), 0)$ has an isolated singular point in the origin. Then (cf. \cite{Milnor}) there exists $\delta_0>0$ such that for any $\delta\in(0,\delta_0]$ there exists $\epsilon_0>0$ such that for any $\epsilon\in (0, \epsilon]$ the map
$$f: \mathcal{N}_{\epsilon(\delta)}=f^{-1}(\mathbb{D}_{\eta})\cap \mathbb{B}_{\delta}\to \mathbb{D}_{\epsilon}^{\ast}$$
composed with $(x_1,\dots, x_k)\mapsto x_1^2+\dots+x_k^2$ is a $C^{\infty}$-fibration called the open Milnor fibration of $f$. The fiber $f^{-1}(\mathbb{S}_{\eta})\cap\mathbb{B}_{\delta}$ over $\eta\in\mathbb{R}^{+}$ is denoted $\mathcal{F}_{\eta}$ and is called the Milnor fibre. By the Regular Value Theorem (see e.g. \cite{Kosinski}) it is a $C^{\infty}$-manifold. 
\subsection{Real Vanishing Cycles} 
Just as we have done in the case of singularities $(\mathbb{R}^n, 0)\to (\mathbb{R}, 0)$ one calls a set of generators of the homology groups of the Milnor fibres a set of vanishing cycles. In the case studied by us in our thesis one has positive and negative vanishing cycles. In our case the base of the fibration is simply connected and this phenomenon does not occur. 

\begin{definition} 
Let $n\in\mathbb{N}$. The real vanishing cycles of $f$ are the generators of $H^n(\mathcal{F}_{\eta}; \mathbb{R})$. 
\end{definition}

There is an action of $\mathbb{S}_{\eta}\subset \mathbb{D}_{\epsilon}^{\ast}$ on $H^n(\mathcal{F}_{\eta}; \mathbb{R})$
defined by lifting $t\mapsto \phi(t)$ via the fibration above.  

\subsection{Notation}
For any $N\in\mathbb{N}$ and for any real analytic submanifold $M\subset \mathbb{R}^{N}$ let $\mathcal{O}_{M}$ denote the sheaf of real analytic functions on $M$.\\
In particular if $(U,\mathcal{O}_U)$ is a complex analytic space with $U\supset M$ an open neighborhood of $M$ then $\mathcal{O}_M=\mathcal{O}_{U|M}$.\\ 
We will only work with real analytic manifolds but for a lucid exposition on real analytic geometry and especially coherent sheaves we refer to the book \cite{italian}.\\
Given a real analytic space $(X, \mathcal{O}_X)$ we say the $X$ is Stein if for any coherent $O_X$-module $\mathcal{C}$, 
$$H^k(X, \mathcal{C})=0,\qquad\forall k>0$$
where the homology is coherent sheaf cohomology.\\
Given real analytic spaces $(X,\mathcal{O}_X)$ and $(Y, \mathcal{O}_Y)$ and a morphism $f: X\to Y$ we will say that $f$ is \emph{Stein} if given a real analytic subspace $Z\subset Y$ which is Stein, then the pullback of $(Z, \mathcal{O}_Z)$ by $f$ is Stein.\\
By the word $\mathcal{D}$-module we will mean a module over the sheaf of differential operators with coefficients real analytic functions. 
Our main reference for $\mathcal{D}$-module theory is the excellent work \cite{Pham} and we follow its terminology for $\mathcal{D}$-modules. For instance $\int_f^{\ast}$ denotes higher direct images in the category of $\mathcal{D}$-modules and $DR_X(\mathcal{O}_X)$ denotes the de Rham complex
$$\Omega_X^n\leftarrow^d \Omega_X^{n-1}\leftarrow^d\dots \leftarrow^d \Omega_X^0=\Omega_X$$
with $d$ the standard exterior derivative on differential $k$-forms. 
\section{$\mathcal{D}$-modules over the real numbers}
\subsection{The Gauss-Manin Connection}
The following holds.

\begin{theorem}\label{main theorem} There exists a ring isomorphism 
$$\left(\int_{f}^k \mathcal{O}_{\mathcal{N}_{\epsilon(\delta)}}\right)_{\eta}\cong H^{k+n}(\mathcal{F}_{\eta}; \mathbb{R})\otimes_{\mathbb{R}} \mathcal{O}_{\mathbb{D}^{\ast}, \eta}$$
\end{theorem}
\begin{proof} 
One the one hand one has by definition of $\int_f^{\cdot}$ that
$$\int_f^k \mathcal{O}_{\mathcal{N}_{\epsilon(\delta)}}=\mathbb{R}f_{\ast}^k(\mathcal{D}_{\mathcal{N}_{\epsilon(\delta)}}\leftarrow \mathbb{D}_{\epsilon}^{\ast}\otimes_{\mathcal{D}_{\mathcal{N}_{\epsilon(\delta)}}}^L \mathcal{O}_{\mathcal{N}_{\epsilon(\delta)}})$$
and on the other hand one has that $DR(\mathcal{D}_{\mathcal{N}_{\epsilon(\delta)}})$ is by the purely algebraic result \cite[Lemma 4.3.5]{Pham} a locally free resolution hence a projective resolution of the transition module $\mathcal{D}_{\mathcal{N}_{\epsilon(\delta)}\leftarrow \mathbb{D}_{\epsilon}^{\ast}}$. And since $f$ is Stein we can apply \cite[Proposition 14.3.4]{Pham} to deduce that the RHS is quasi-isomorphic to
$$\mathbb{R}^{k+n}f_{\ast}DR_{\mathcal{N}_{\epsilon(\delta)}/\mathbb{D}_{\epsilon}^{\ast}}(\mathcal{O})=\mathbb{R}^{k+n} f_{\ast}(\Omega_{\mathcal{N}_{\epsilon(\delta)}/\mathbb{D}_{\epsilon}^{\ast}}^{\cdot}).$$
Using the relative Poincaré Lemma (\cite{Kulikov}[section 3.3]) then gives 
$$\int_f^k \mathcal{O}_{\mathcal{N}_{\epsilon(\delta)}}=H^{k+n}(\mathcal{N}_{\epsilon(\delta)}/\mathbb{R}_{\epsilon}^{\ast}; \mathbb{R})$$
in the notation of \cite{Pham}. Since $f_{|\mathcal{N}_{\epsilon(\delta)}}$ is a locally trivial $C^{\infty}$-fibration this sheaf is locally constant of fiber $H^{k+n}(f_{|\mathcal{N}_{\epsilon(\delta)}^+}^{-1}(\eta); \mathbb{R})$ over $\eta\in \mathbb{D}_{\epsilon}^{\ast}$. This finishes the proof.
\end{proof}
The theorem only gives information about $H^n(\mathcal{F}_{\eta}; \mathbb{R})$ which is quite unknown. There are very few general results about singularities of maps of codimension larger than the unity. From now onwards, we shall for brevity sometimes omit the subindices and write simply $\mathcal{N}$ for the total spaces of the real Milnor fibration.
We now define the Gauss-Manin system associated to our singularity.
\begin{definition}
Let $k\in\mathbb{N}$. The $k$-th Gauss-Manin system associated to $f$ is the $\mathcal{O}_{\mathbb{D}_{\epsilon}^{\ast}}$-module $\int_f^k \mathcal{O}_{\mathcal{N}}.$
\end{definition}

Let us recall from \cite{Pham} that the systems $\int_f^k \mathcal{O}_{\mathcal{N}}$ are as pushforwards naturally endowed with a $\mathcal{D}$-module structure.

\subsection{Real Analytic Solutions}
The space of solutions to a $\mathcal{D}$-module is defined by $\mathcal{D}$-module homomorphisms into the sheaf of regular functions.

\begin{definition} The space of solutions of a $\mathcal{D}$-module $\mathfrak{d}$ over $X$ is
$Sol(\mathfrak{d})=Hom_{\mathcal{D}}(\mathfrak{d}, \mathcal{O}_X)$
\end{definition}

In the case where the $\mathcal{D}$-module is a Gauss-Manin system one constructs solutions to $\int_f^0 \mathcal{O}_{\mathcal{N}_{\epsilon(\delta)}^{+}}$ as follows. Let $U\subset\mathbb{D}_{\epsilon}^{\ast}$ be open and define
$$(\int_f^k \mathcal{O}_{\mathcal{N}_{\epsilon(\delta)}})(U) \to \mathcal{O}_{\mathbb{D}_{\epsilon}^{\ast}}(U)$$
$$c\mapsto \int_{h(\eta)} c$$
where $h(\eta)\in H^n(\mathcal{F}_{\eta}; \mathbb{R})$.
\begin{corollary}\label{corr one}
The application $h\mapsto \int_h$ is an isomorphism
$$H^n(\mathcal{F}_{\eta};\mathbb{R})\to Hom_{\mathcal{D}}(\int_f^0 \mathcal{O}_{\mathcal{N}_{\epsilon(\delta)}}, \mathcal{O}_{\mathbb{D}_{\epsilon}^{\ast}})$$
\end{corollary}
\begin{proof}
Use de Rham's theorem, see \cite[Proposition 15.1]{Pham}.
\end{proof}

In other words, one can identify the highest degree cohomology group of the Milnor fiber with the real analytic solutions of the Gauss-Manin system associated to the singularity.

\begin{remark}
In the holomorphic case one can replace the word 'solution' above with 'horisontal solution'. It is not clear for us whether or not this might be done in the real case.
\end{remark}

\section{Polynomial Dynamical Systems in the Plane}
\subsection{Hilbert's 16th Problem}
Given the system of differential equations
$$(\dot{x}_1(t), \dot{x}_2(t))=(P(x_1(t), x_2(t)), Q(x_1(t), x_2(t)))\qquad(\ast)$$
with $P, Q\in \mathbb{R}[x_1, x_2]$ real polynomials recall that a \emph{limit cycle} is a periodic solution of $(\ast)$ which is the limit of at least one other solution. Hilbert's 16th problem then asks for how many limit cycles a system $(\ast)$ can have. 
\subsection{Localising a dynamical system}
Let us write $V(x_1, x_2)=(P(x_1, x_2), Q(x_1, x_2))$. Recall that each periodic solution of a dynamical system contains in it's interior a stationary point. Let $Crit(V)$ denote the set of stationary points of $V$ and let us assume that there are only finitely many stationary points; 
$$Crit(V)=\{p_1,\dots, p_k\}$$
At each point $p_i$ let $f_{i}=0,$ define the germs of solutions of $(\ast)$. In detail one considers a sufficiently small ball $\mathbb{B}_{\delta_i}\ni p_i$ and one takes a local equation for the system of differential equations 
$$rank \nabla V(x_1, x_2)\leq 1,\qquad (x_1, x_2)\in \mathbb{B}_{\delta_i}$$ 
or equivalently
$$\frac{\partial P}{\partial x_1}\frac{\partial Q}{\partial x_2}-\frac{\partial P}{\partial x_2}\frac{\partial Q}{\partial x_1}=0,\qquad (x_1, x_2)\in\mathbb{B}_{\delta_i}$$
which one integrates giving the local equation $f_{i}$. Then $f_{i}: (\mathbb{R}^2, p_i)\to (\mathbb{R}^2, f_{i}(p_i)$ defines $(n_1+\dots+n_k)$ singularities. Let
$$\mathfrak{g}_{i}=\int_{f_{i}}^0 \mathcal{O}_{\mathcal{N}(\epsilon_i(\delta_i))}$$
denote the corresponding Gauss-Manin systems where 
$$f_{i}: \mathcal{N}_{\epsilon_i(\delta_i)}\to \mathbb{D}_{\epsilon_i}^{\ast}$$
are the Milnor fibrations of $f_{i}$ with fibers $\mathcal{F}_{\eta_i}^{f_{i}}$ over $\eta_i\in\mathbb{D}_{\epsilon_i}^{\ast}$. By the \hyperref[main theorem]{Theorem \ref*{main theorem} } there is an isomorphism of rings 
$$(\mathfrak{g}_{i})_{f_{i}(p_i)}\to H^n(\mathcal{F}_{\eta_i}^{f_{i}}; \mathbb{R})$$
and there is a basis of vanishing cycles
$$[\gamma_{il}(t)]\in H^n(\mathcal{F}_{\eta_i}^{f_{i}}; \mathbb{R})$$
\subsection{The Solution}
The dynamical system $(\ast)$ is a $\mathbb{D}$-module $\mathfrak{f}=d/dt-V$. Its real analytic solutions $Sol(\mathfrak{f})$ therefore forms a real vector space \cite{coutinho}. Its set of periodic solutions forms a real sub-vectorspace $PSol(\mathfrak{f})$. The same is true if we localise as above and consider local solutions $Sol_i(\mathfrak{f})$ and local periodic solutions $PSol_i(\mathfrak{f})$ in the local Milnor ball $\mathbb{B}_{\delta_i}$.
In the local Milnor ball $\mathbb{B}_{\delta_i}$ we notice that $V(x(t))=dx/dt$ is equivalent to
$$V(x(t))=\eta_i(t),\qquad dx/dt=\eta_i(t)\qquad(\ast\ast)$$
Consider
$$V(x(t))=\eta_i(t)\qquad dx/dt=\eta_i(t),\qquad \lVert \eta_i(t)\rVert =\eta_i \qquad (\ast\ast\ast)$$
where we consider $\eta_i(t)$ as a path $\eta_i: \mathbb{R}\to\mathbb{R}^2$. If $x(t)$ is a periodic solution to $(\ast\ast)$ it describes a circle which we can parametrise such that the speed is constant. But then $x(t)$ satisfies $(\ast\ast\ast)$. 
In particular every element in $PSol_i(\mathfrak{f})$ is a periodic path $x(t)$ in the Milnor fibre $\mathcal{F}_{\eta_i}^{f_i}$. Taking homology these paths are spanned $[x(t)]=\sum a_i [\gamma_l(t)]$ by the vanishing cycles and forms the vector space $H_n(\mathbb{F}_{\eta_i}^{f_i}; \mathbb{R})$. Sending $x(t)\mapsto [x(t)]$ therefore gives an injection
$$PSol_i(\mathfrak{f})\to H_n(\mathcal{F}_{\eta_i}^{f_i}; \mathbb{R})$$
and hence into the solutions $Sol(\mathfrak{g}_i)$ of the local Gauss-Manin system. Write $l_{i}=rank H^n(\mathcal{F}_{\eta_i}^{f_{i}}; \mathbb{R})$.

\begin{theorem} There are at most  $\sum_{i=1}^k l_{i}$ limit cycles of $(\ast)$ in homology. If $\pi\circ f_i, i=1,\dots, k$ with $\pi: \mathbb{R}^2\to\mathbb{R}$ a standard projection is a submersion restricted to the total space of the Milnor fibration of $f_i$ then equality holds.
\end{theorem}
\begin{proof} Each limit point is a periodic solution hence corresponds to a stationary point $p_i$, $i=1,\dots, k$. Each stationary point corresponds to a germ of singularity $f_{i}$ to which we have associated a $\mathcal{D}$-module $(\mathfrak{g}_{i})_{\eta_i}$. There is a bijection between the set of solutions to this $\mathcal{D}$-module and the vanishing cycles by \hyperref[corr one]{Corollary \ref*{corr one} }. Each vanishing cycle is equipped with an action $t\mapsto [\gamma(t)], t\in \mathbb{S}_{\eta_i}\subset\mathbb{R}^2$ which has a certain period $\alpha$. On the other hand $Sol(\mathfrak{g}_{i})$ is equipped with an action
$t\mapsto \int_{\gamma(t)}$
Let $\mathcal{L}\subset Sol(\mathfrak{g}_{i})$ denote the set of periodic solutions i.e images of the isomorphism which are periodic under the above action. Then the isomorphism of \hyperref[corr one]{Corollary \ref*{corr one} }, call it $\phi$, gives $\mathcal{L}\cong \phi^{-1}(\mathcal{L})\subset H^{n}(\mathcal{F}_{\eta_i}^{f_{i}}; \mathbb{R})$ whence the upper bound. For the lower bound we need to show that there exists for each vanishing cycle a trajectory tending to it from the inside or outside. By our assumptions the Milnor fibration
$$\mathcal{N}_{\epsilon_i(\delta_i)}\to^{f_i} \mathbb{D}_{\epsilon_i}\to^{(x_1, x_2)\to x_1^2+x_2^2} \mathbb{R}^{+}$$
contains as a restriction the $C^\infty$-fibration
$$\mathcal{N}_{\epsilon_i(\delta_i)}\to^{\pi\circ f_i} \mathbb{D}_{\epsilon_i}\to^{(x_1, x_2)\mapsto x_1^2+x_2^2} \mathbb{R}^{+}.$$
Let the projection be the projection to the first coordinate. 
Let $\gamma\in \mathcal{F}_{\eta_i}^{f_i}$ be a vanishing cycle and let $\rho(t)=f\circ \gamma(t).$ We can assume $\rho: t\mapsto \eta_i e^{2\pi i t/\eta_i}$. In particular $\rho$ is defined and analytic in the disc $\mathbb{D}_{\eta_i}$. 
Let $P= \{t\cup \rho(t), t^2\in (0,\eta_i]\}$. Then, by triviality of the Milnor fibration and by our hypothesis there exists a connected component of $\Gamma=\{\gamma(t): f\circ \gamma(t)\in P\}$. Again by triviality this component is diffeomorphic to an interval and hence is a curve $\tilde{\gamma}(t)$. This curve is by \hyperref[corr one]{Corollary \ref*{corr one} } for each $t$ a solution hence gives a by smoothening the corners if necessary a trajectory tending to $\gamma$ as $t\to \pm \eta $ which proves the equality.
\end{proof}

\begin{remark}
    The ideas in the proof works more generally for real analytic vector fields on $\mathbb{R}^n$
 and dynamical systems $\dot{x}(t)=V(x(t))$. Then there is for $k>1$ a bijective correspondence between vanishing cycles of a certain system of germs of singularities and the limit cycles of the dynamical system. In case $k=1$ one could speak of positive and negative limit cycles, corresponding to positive and negative vanishing cycles. For instance $\dot{x}(t)=x-x^3$ has at $x^2=1$ two positive limit cycles and no negative limit cycles and at $x=0$ a positive limit cycle and a negative limit cycle. There is a diffeomorphism $x\mapsto -x$ between the "limit cycles" at $x=0.$
\end{remark}
\begin{remark}
    One can associate to each real analytic singularity $f: (\mathbb{R}^n, 0)\to (\mathbb{R}^k, 0)$
the conservative dynamical system $\dot{x}(t)=\nabla f(x(t))$. The vanishing cycles of $f$ corresponds to the periodic orbits of the dynamical system. This is in essence what Newton does in Proposition XVI Theorem VIII of the Principia \cite{newton} but he of course reduced it to a problem about Euclidean geometry\footnote{which strictly speaking concerns non-linear planes and curves, it's just that these are studied essentially through linearisation which in a more refined form Newton of course created by adding to Euclids axioms a Lemma, which in is really both an axiom and a definition of the infinitesimal} just as he reduced physics to geometry. 
\end{remark}
\begin{definition}
    A morsification of a polynomial dynamical system $\dot{x}(t)=V(x(t))$ is a dynamical system $\dot{x}(t)=V_s(x(t))$ such that $V_s(x)=\tilde{V}(x, s)$ is polynomial, $V_0(x)=V(x)$ and such that for an open dense subset of values of the parameter $s$ if $V_s(p)=0$ the vector field $V_s$ has only nondegenerate isolated critical points. 
\end{definition}
\begin{remark}
    Since the vanishing cycles of the local singularities $f_{i,s}$ associated to the morsification of a dynamical system $\dot{x}(t)=V(x(t))$ are the same as the vanishing cycles of $f_i$, the theorem above says that under its (very restrictive) second assumption morsification leaves limit cycles invariant.
\end{remark}

\bibliographystyle{plain}
\bibliography{main.bib}

\begin{thebibliography}{1}

\bibitem{coutinho}
S.~C. Coutinho.
\newblock {\em A Primer of Algebraic D-Modules}.
\newblock London Mathematical Society Student Texts. Cambridge University Press, 1995.

\bibitem{italian}
{P}atrizia~{M}acr{\`i} Francesco~{Guaraldo} and {A}llessandro {T}ancredi.
\newblock {\em Topics on {R}eal {A}nalytic {S}paces}.
\newblock Advanced {L}ectures in {M}athematics. Vieweg \& {T}eubner {V}erlag, 1986.

\bibitem{Kosinski}
Antoni~A. Kosinski.
\newblock {\em Differential {M}anifolds}.
\newblock Academic Press, Inc. Boston, MA, 1993.

\bibitem{Man58}
Yu.{I}. {M}anin.
\newblock Algebraic curves over fields with differentiation.
\newblock {\em Izv. {A}kad. {N}auk {S}{S}{S}{R} {S}er. {M}at.}, 22(6):737--756, 1958.

\bibitem{Milnor}
John Milnor.
\newblock {\em Singular {P}oints of {C}omplex {H}ypersurfaces.}
\newblock Princeton University Press, Princeton, N.J, 1968.

\bibitem{newton}
{I}saac Newton.
\newblock {\em Philosophiae {N}aturalis {P}rincipia {M}athematica}.
\newblock Societas Regiae, Londini, 1687.

\bibitem{Pham}
{F}rédéric Pham.
\newblock {\em {S}ingularités des {S}ystèmes {D}ifferentiels de {G}auss-Manin}, volume~2 of {\em Progress in {M}athematics}.
\newblock Springer, 1976.

\bibitem{Kulikov}
{V}alentine {S}.~Kulikov.
\newblock Mixed {H}odge {S}tructures and {S}ingularities.
\newblock {\em Cambridge {T}racts in {M}athematics}, 132, 1998.

\end{thebibliography}

\end{document}